\documentclass[12pt,leqno]{amsart}

\usepackage[T1]{fontenc}
\usepackage[utf8]{inputenc}
\usepackage{geometry}
\usepackage{amsmath,amssymb,amsthm,mathtools}
\usepackage{booktabs}
\usepackage{array}
\usepackage{capt-of}
\usepackage{microtype}
\usepackage{enumitem}
\usepackage[colorlinks=true,allcolors=black]{hyperref}

\numberwithin{equation}{section}
\numberwithin{table}{section}
\setlist{topsep=3pt,itemsep=2pt,parsep=0pt}

\newtheorem{theorem}{Theorem}[section]
\newtheorem{proposition}[theorem]{Proposition}
\newtheorem{lemma}[theorem]{Lemma}
\newtheorem{corollary}[theorem]{Corollary}
\newtheorem{conjecture}[theorem]{Conjecture}
\newtheorem{question}[theorem]{Question}

\theoremstyle{definition}
\newtheorem{definition}[theorem]{Definition}
\newtheorem{example}[theorem]{Example}
\newtheorem*{ack}{Acknowledgments}
\theoremstyle{remark}
\newtheorem{remark}[theorem]{Remark}

\DeclareMathOperator{\Nef}{Nef}
\DeclareMathOperator{\N}{N}
\DeclareMathOperator{\G}{G}
\DeclareMathOperator{\Pic}{Pic}
\DeclareMathOperator{\Hom}{Hom}
\newcommand{\ZZ}{\mathbb Z}
\newcommand{\QQ}{\mathbb Q}
\newcommand{\RR}{\mathbb R}
\newcommand{\PP}{\mathbb P}
\newcommand{\cH}{\mathcal H}
\newcommand{\cR}{\mathcal R}
\newcommand{\tV}{\widetilde V}
\newcommand{\tB}{\widetilde B}
\newcommand{\NEbar}{\operatorname{NE}}
\newcommand{\gap}{\delta}

\title[The gap between Gongyo indices of toric Fano varieties]
{On the gap between the integral and rational Gongyo indices of toric Fano varieties}
\author[Hiroshi Sato]{Hiroshi Sato}

\date{}


\address{Department of Applied Mathematics,
Faculty of Science, Fukuoka University,
8-19-1, Nanakuma, Jonan-ku, Fukuoka 814-0180, Japan}
\email{hirosato@fukuoka-u.ac.jp}

\hypersetup{
  pdftitle={On the gap between the integral and rational Gongyo indices of toric Fano varieties},
  pdfauthor={Hiroshi Sato},
  pdfsubject={The gap between Gongyo indices of toric Fano varieties},
  pdfkeywords={toric variety, Fano variety, weak Fano variety, Gongyo index, pseudo-symmetric variety}
}

\begin{document}

\begin{abstract}
Exact calculations for smooth toric Fano varieties in dimensions
at most seven suggest a sharp dimension-dependent lower bound
for the positive difference between the rational and integral
Gongyo indices.
We compute both indices for the smooth toric Fano models obtained
from blow-ups of projective space at torus-invariant points by
the anticanonical minimal model program.
For pseudo-symmetric smooth toric Fano varieties, we prove the
proposed lower bound and determine all equality cases.
We also apply a dimension-raising construction to obtain an explicit
infinite family of smooth toric Fano varieties with positive gap from
odd-dimensional weak Fano varieties with zero gap.
For the standard models considered, we determine the smallest
dimension increase needed to obtain a Fano variety by iterating
this construction and compute both indices of the minimal lifts.
Finally, every positive rational number occurs as a Gongyo-index
gap, even within the pseudo-symmetric class.
\end{abstract}

\maketitle

\section{Introduction}

For a smooth Fano variety $X$, Gongyo introduced the integral total
index $\tau_X(\ZZ)$ and asked whether it always agrees with
its rational counterpart $\tau_X(\QQ)$; in
this paper we call them the integral and rational Gongyo indices.
Both indices measure how large the sum of coefficients can be
in a decomposition of the anticanonical class into nonzero nef
integral divisor classes, using integral or rational coefficients,
respectively.
The first
negative example was found by Atsushi Ito in the non-toric setting: for the
degree-five del Pezzo surface
$S_5=\operatorname{Bl}_{p_1,\ldots,p_4}\PP^2$
with the four points in general position,
the two indices are $2$ and $5/2$,
respectively. This calculation is recorded in
\cite[Remark~5.8]{enwright-et-al-nef-complexity}; see also
\cite[Theorem~5.7]{enwright-et-al-nef-complexity} for the complete
calculation for smooth del Pezzo surfaces.
Gagliardi--Hofscheier--Pearson
subsequently answered the toric version negatively
by exhibiting the
Voskresenskij--Klyachko fourfold $V^4$: they proved
$\tau_{V^4}(\ZZ)=2$ and constructed
a rational decomposition of total
weight $5/2$ \cite[Example~5.4 and Question~5.3]{gagliardi-hofscheier-pearson}.
Our four-dimensional computation recovers this example,
proves the exact
value $\tau_{V^4}(\QQ)=5/2$,
and shows that $\tV^4$ is the only other smooth
toric Fano fourfold with positive gap.

Exact computations over the complete classifications
of smooth toric Fano varieties
in dimensions at most
seven show that
no positive difference occurs in dimensions at most three,
that the smallest positive difference is $1/2$
in dimensions four and five,
and that it is $1/3$ in dimensions six and seven.
Among the $80{,}892$
varieties in these seven classifications,
only $818$ have unequal Gongyo
indices; see Theorem~\ref{thm:low-dimensional-computation}.
Motivated by these computations,
we propose the following conjecture.

\begin{conjecture}[Minimal positive gap]
\label{conj:minimal-positive-gap}
Let $X$ be a smooth toric Fano variety of dimension $d\ge4$.
If
$\tau_X(\QQ)\neq\tau_X(\ZZ)$, then
\[
 \tau_X(\QQ)-\tau_X(\ZZ)
 \geq \frac{1}{\lfloor d/2\rfloor}.
\]
\end{conjecture}
One of the main results of this paper
is Theorem~\ref{thm:main}, which proves
Conjecture~\ref{conj:minimal-positive-gap}
for pseudo-symmetric smooth toric Fano varieties
and determines
all equality cases in this class.
In particular, the proposed lower bound is attained
in every dimension at least four.

To study these indices under birational modifications and the
dimension-raising construction used below, we work more generally
with a nonzero nef integral divisor class $D$.
We therefore define the integral and rational Gongyo indices
of the pair $(X,D)$.
The anticanonical specialization recovers
the usual invariants whenever
$-K_X$ is nef.
The Hilbert-basis formulation, upper bounds from linear functionals,
and product additivity all hold in this generality.
In particular, if the
nef cone is unimodular,
then the two indices agree for every such $D$.
We prove that this applies to
every generalized Bott manifold and to every
smooth projective toric variety of Picard number at most three.

A second part of the paper concerns the toric models constructed in
\cite{sato-tsuzuki-antiflips} from blow-ups of $d$-dimensional
projective space at $n$ torus-invariant points.
In the
range $2n-1<d$, the resulting model is a generalized Bott manifold and hence has no
gap.
In the critical range $2n-1\geq d$, the even- and odd-dimensional
models with common parameters $m\ge2$ and $m+1\le n\le2m+1$
have the same nef semigroup under the coordinates introduced below,
but different anticanonical classes.
We compute both branches, including the additional odd endpoint
$n=2m+2$. For $m\ge2$ and $n\ge m+2$, the $2m$-dimensional
models have gap $1/m$, as in the del Pezzo and pseudo-del Pezzo
families, whereas the $(2m+1)$-dimensional models are weak Fano
and have zero anticanonical gap.
We also prove that the anticanonical Gongyo indices are
invariant under toric flops between smooth projective
toric weak Fano varieties.

The odd-dimensional weak Fano models, despite having zero gap,
give rise to an explicit infinite family of smooth toric Fano
varieties with positive gap.
We use the dimension-raising construction introduced in
\cite{sato-index-two}, and show that it preserves the additive
structure of nef integral divisor classes while changing the
anticanonical class.
For each fixed standard model in the range
$m+2\le n\le2m+2$, we determine the smallest dimension
increase that yields a Fano variety by iterating this
construction and compute both Gongyo indices of the
resulting minimal lifts; see
Theorems~\ref{thm:ST-minimal-lift} and
\ref{thm:ST-lift-indices}.

Finally, the formulas for the del Pezzo and pseudo-del Pezzo factors,
together with Ewald's decomposition theorem and product additivity,
give both indices explicitly for pseudo-symmetric varieties and show
that every positive rational number occurs as a Gongyo-index gap
within this class.

Section~\ref{sec:total-index} develops the general formalism, including flop
invariance.  Section~\ref{sec:bott} treats generalized Bott manifolds and
splitting fans, and Section~\ref{sec:small-picard} treats Picard number at
most three.  Section~\ref{sec:ST} studies the toric models described
above and their Fano lifts.  Section~\ref{sec:pseudosym} proves the pseudo-symmetric
results.  The exact computations through dimension seven are collected in
Appendix~\ref{sec:computations}.

\begin{ack}
The author thanks Professor Yoshinori Gongyo for raising a question
about total indices in the toric setting several years ago
and for recent advice on this work.
The author is especially grateful to Professor
Osamu Fujino for drawing
the author's attention to the question of whether the integral
and rational Gongyo indices agree in the toric setting,
and for valuable advice and suggestions throughout the
preparation of this paper.

The author used ChatGPT (OpenAI) for assistance with developing
and checking mathematical arguments, searching the literature,
drafting and revising the manuscript, and developing and debugging
the computational and verification code.
The author takes full responsibility for the content of this
paper and its computational supplement.

The author was partially supported by JSPS KAKENHI Grant
Number JP24K06679.
\end{ack}

\section{Gongyo indices and toric computation}\label{sec:total-index}
All toric varieties below are defined over an algebraically closed field.
For a fan in a lattice $N$, we write $M=\Hom(N,\ZZ)$ for the dual lattice.

Gongyo's original invariant is attached to the anticanonical class of a Fano
variety \cite{gongyo-total-index}.  The integral and rational versions, and
Gongyo's equality question, are discussed in
\cite[Definition~5.1 and Question~5.3]{gagliardi-hofscheier-pearson}.  We use
the following slightly broader toric formulation because it makes the effect
of the dimension-raising construction in
Subsection~\ref{subsec:dimension-raising} transparent.

Throughout, an \emph{integral divisor}
means an integral Cartier divisor class,
and $D\neq0$ means a nonzero class in $\Pic(X)$.

\begin{definition}\label{def:total-index}
Let $X$ be a smooth projective toric variety, let
$0\neq D\in\Nef(X)\cap\Pic(X)$,
and let $\mathbb K\in\{\ZZ,\QQ\}$.  Set
\[
 \tau_X(D;\mathbb K)
 :=\sup\left\{
      \sum_{i=1}^r a_i
      \, \left|\,
       D=\sum_{i=1}^r a_iD_i,\
       0\neq D_i\in\Nef(X)\cap\Pic(X),\
       a_i\in\mathbb K_{>0}
     \right.\right\}.
\]
For $\mathbb K=\QQ$, the equality in the definition is taken in
$\Pic(X)\otimes_{\ZZ}\QQ$.
We call $\tau_X(D;\ZZ)$ and $\tau_X(D;\QQ)$ the
\emph{integral Gongyo index} and the \emph{rational Gongyo index} of the pair
$(X,D)$, respectively; equivalently, they are the Gongyo indices of $X$ with
respect to $D$.  Their difference
\[
 \delta_X(D):=\tau_X(D;\QQ)-\tau_X(D;\ZZ)\geq0
\]
is the \emph{Gongyo-index gap} of the pair $(X,D)$.

If $-K_X$ is nef and nonzero, we use the abbreviations
$\tau_X(\mathbb K):=\tau_X(-K_X;\mathbb K)$ and
$\delta(X):=\delta_X(-K_X)$.
Thus, for a smooth toric Fano variety these are precisely the usual integral
and rational Gongyo indices and their gap.
We use the same anticanonical
terminology for a smooth toric weak Fano variety.
\end{definition}

The extension from $D=-K_X$ to an arbitrary nef integral class is a
convention of the present paper.  Its role is auxiliary: the main questions
below concern the anticanonical specialization.
For a smooth complex toric Fano variety $X$, the nef complexity
$c_X$ of \cite[Definition~1.2]{enwright-et-al-nef-complexity} satisfies
$c_X=\dim X+\rho(X)-\tau_X(\QQ)$.
Indeed, numerical and linear equivalence of Cartier divisors agree
on a smooth complete toric variety.
Thus the rational index is already encoded by nef complexity;
here we study its difference from the integral index and compute
both indices for the toric families considered below.

\begin{remark}[Relation with classical nef-partitions]
A closely related classical construction is the theory of nef-partitions,
introduced by Borisov and developed
by Batyrev--Borisov and Batyrev--Nill in
connection with mirror symmetry for Calabi--Yau
complete intersections
\cite{borisov-nef-partitions,batyrev-borisov-complete-intersections,
batyrev-nill-mirror}. In particular, Batyrev and Nill proved that the
length of a proper nef-partition of
a $d$-dimensional reflexive polytope is at
most $2d$ \cite[Proposition~6.16]{batyrev-nill-mirror}.

Let $X$ be a smooth toric Fano variety and let $\Delta_{-K_X}$ be its
anticanonical polytope.  Every proper centered nef-partition of
$\Delta_{-K_X}$ gives an admissible integral decomposition of $-K_X$ and
therefore contributes its length to $\tau_X(\ZZ)$.
Conversely, after expanding integral coefficients, an admissible
divisor decomposition gives a lattice Minkowski decomposition of
$\Delta_{-K_X}$.
The classical definition additionally requires lattice points in
the summands whose sum is the origin.
For general reflexive polytopes this is an additional condition;
see \cite[Definition~3.1, Remark~3.5, and Example~3.8]{batyrev-nill-mirror}.
The maximal length of a proper nef-partition is therefore at most
$\tau_X(\ZZ)$; we do not use an identification
between these notions here.
\end{remark}

Let $L$ be a lattice,
let $C\subset L_{\RR}$ be a strongly convex rational
polyhedral cone,
and put $\Gamma=C\cap L$.
By Gordan's lemma,
$\Gamma$ is a finitely generated affine
semigroup; see \cite[Section~1.2]{cox-little-toric}.
We recall the standard notion of its Hilbert basis.

\begin{definition}[Hilbert basis]\label{def:hilbert-basis}
The \emph{Hilbert basis} $\cH(C,L)$ is the unique minimal finite generating
set of the additive monoid $\Gamma$.  Equivalently,
\[
 \cH(C,L)
 =\left\{h\in\Gamma\setminus\{0\}\,\left|\,
      h\neq u+v\text{ for all }u,v\in\Gamma\setminus\{0\}
      \right.\right\}.
\]
We also write $\cH(\Gamma)=\cH(C,L)$.
\end{definition}

\begin{lemma}[Basic properties of Hilbert bases]\label{lem:hilbert-basic}
Write $\cH(C,L)=\{h_1,\ldots,h_N\}$.  Then:
\begin{enumerate}[label=\textup{(\roman*)}]
 \item every element of $\Gamma$ is a nonnegative integral combination of
 $h_1,\ldots,h_N$;
 \item the primitive generator of every extremal ray of $C$ belongs to
 $\cH(C,L)$;
 \item if the primitive extremal-ray generators form a $\ZZ$-basis of $L$,
 then they are the Hilbert basis and every element of $\Gamma$ has a unique
 expression in them;
 \item for two affine semigroups
 $\Gamma_1=C_1\cap L_1$ and $\Gamma_2=C_2\cap L_2$,
 \[
  \cH(\Gamma_1\oplus\Gamma_2)
  =(\cH(\Gamma_1)\times\{0\})\cup(\{0\}\times\cH(\Gamma_2));
 \]
 \item for every $b\in\Gamma$, the polyhedron
 \[
  P_b:=\left\{x=(x_1,\ldots,x_N)\in\RR_{\geq0}^N\ \middle|\
              \sum_{j=1}^N x_jh_j=b\right\}
 \]
 is bounded.
\end{enumerate}
\end{lemma}

\begin{proof}
Parts~\textup{(i)}--\textup{(iii)}
follow directly from the generation,
extremality, and lattice-basis properties.
For~\textup{(iv)}, an element
with two nonzero components is decomposable,
while an element supported in
one component is indecomposable exactly when that component is.
Finally, if $P_b$ were unbounded, it would contain a ray
$\{x+tu\mid t\ge0\}$ with $u=(u_1,\ldots,u_N)\ne0$.
Since every point of this ray has nonnegative coordinates,
we have $u_j\ge0$ for all $j$.
The defining equation of $P_b$ gives $\sum_j u_jh_j=0$.
Choose $j_0$ with $u_{j_0}>0$. Then
$-h_{j_0}=\sum_{j\ne j_0}(u_j/u_{j_0})h_j\in C$,
contradicting the strong convexity of $C$.
\end{proof}

For a smooth projective toric variety $X$, put
$\Gamma_X:=\Nef(X)\cap\Pic(X)$, and put
\[
 \cH_X:=\cH\left(\Nef(X),\Pic(X)\right)=\{h_1,\ldots,h_N\}.
\]
Gagliardi--Hofscheier--Pearson observed, for the anticanonical specialization,
that passage to the Hilbert basis turns the supremum into a maximum
\cite[Remark~5.5]{gagliardi-hofscheier-pearson}.  The same proof gives the
following form for every target $D$.

\begin{proposition}[Hilbert-basis formulation]
\label{prop:hilbert-formulation}
Let $X$ be a smooth projective toric variety and
$0\neq D\in\Gamma_X$.  Then
\begin{align*}
 \tau_X(D;\ZZ)
 &=\max\left\{\sum_{j=1}^N z_j\ \middle|\
      \sum_{j=1}^Nz_jh_j=D,\ z=(z_1,\ldots,z_N)\in\ZZ_{\geq0}^N\right\},\\
 \tau_X(D;\QQ)
 &=\max\left\{\sum_{j=1}^N x_j\ \middle|\
      \sum_{j=1}^Nx_jh_j=D,\ x=(x_1,\ldots,x_N)\in\QQ_{\geq0}^N\right\}.
\end{align*}
In particular, the suprema in Definition~\ref{def:total-index} are maxima.
\end{proposition}

\begin{proof}
Write every $D_i$ in an admissible decomposition
$D=\sum_{i=1}^r a_iD_i$ as $D_i=\sum_j n_{ij}h_j$,
with $n_{ij}\in\ZZ_{\ge0}$.
Replacing each $D_i$ by this Hilbert-basis decomposition
does not decrease the total coefficient sum,
since $\sum_j n_{ij}\ge1$.
Collecting the coefficients of each $h_j$ then gives
nonnegative coefficients satisfying the corresponding
displayed conditions, with total sum at least $\sum_i a_i$.
Thus it suffices to consider decompositions
involving only Hilbert-basis elements.
Conversely, any coefficients satisfying those conditions
give an admissible decomposition.
It remains to show that both suprema are attained.
By Lemma~\ref{lem:hilbert-basic}~(i) and~(v), the polyhedron $P_D$
is bounded and contains an integral point.
Thus $P_D\cap\ZZ^N$ is finite and nonempty.
Moreover, the linear function $\sum_j x_j$
attains its maximum on $P_D$ at a vertex,
which is rational since $P_D$ is defined over $\QQ$.
Hence both suprema are attained, proving the displayed formulas.
\end{proof}

\begin{corollary}[Upper bounds from linear functionals]
\label{cor:dual}
With the notation above, let
$\lambda\in\Pic(X)^\vee_{\RR}$
satisfy $\lambda(h_j)\ge1$ for $1\le j\le N$.
Then $\tau_X(D;\QQ)\le\lambda(D)$.
\end{corollary}

\begin{proof}
For any expression $D=\sum_{j=1}^N x_jh_j$
with $x_j\in\QQ_{\ge0}$, we have
\[
\sum_{j=1}^N x_j
\le \sum_{j=1}^N x_j\lambda(h_j)
= \lambda(D).
\]
The assertion follows from
Proposition~\ref{prop:hilbert-formulation}.
\end{proof}

\begin{remark}
For each $0\ne D\in\Gamma_X$, there exists
$\lambda\in\operatorname{Pic}(X)^\vee_{\mathbb R}$
such that $\lambda(h_j)\ge1$ for all $j$
and $\lambda(D)=\tau_X(D;\QQ)$.
This follows from Proposition~\ref{prop:hilbert-formulation}
and the strong duality
theorem for linear programming;
see, for example, \cite[Sections~5.2.3--5.2.4]{BV04}.
Only the upper bound in Corollary~\ref{cor:dual} is needed below.
\end{remark}

\begin{definition}\label{def:unimodular-cone}
Let $L$ be a lattice and let $C\subset L_{\RR}$ be a full-dimensional
strongly convex rational polyhedral cone.  We say that $C$ is
\emph{unimodular with respect to $L$} if its primitive extremal-ray
generators form a $\ZZ$-basis of $L$.
\end{definition}

\begin{proposition}[Unimodular nef cones]\label{prop:unimodular-gap-zero}
Let $X$ be a smooth projective toric variety whose nef cone is unimodular
with primitive generators $L_1,\ldots,L_\rho$.  If
\[
 0\neq D=b_1L_1+\cdots+b_\rho L_\rho\in\Gamma_X,
\]
then $b_i\in\ZZ_{\geq0}$ and
\[
 \tau_X(D;\ZZ)=\tau_X(D;\QQ)=b_1+\cdots+b_\rho.
\]
In particular, $\delta_X(D)=0$ for every nonzero nef integral class $D$.
\end{proposition}

\begin{proof}
The $L_i$ form a lattice basis and generate the nef cone, so
$\Gamma_X=\bigoplus_i\ZZ_{\geq0}L_i$.  Its Hilbert basis is exactly
$\{L_1,\ldots,L_\rho\}$,
and the equation $D=\sum_i x_iL_i$ has the unique
solution $(x_1,\ldots,x_\rho)=(b_1,\ldots,b_\rho)$.
\end{proof}

\begin{proposition}[Product additivity]\label{prop:product}
Let $X,Y$ be smooth projective toric varieties and let
$0\neq D\in\Gamma_X$, $0\neq E\in\Gamma_Y$.  For
$\mathbb K\in\{\ZZ,\QQ\}$,
\[
 \tau_{X\times Y}(p_X^*D+p_Y^*E;\mathbb K)
 =\tau_X(D;\mathbb K)+\tau_Y(E;\mathbb K),
\]
where $p_X:X\times Y\to X$ and $p_Y:X\times Y\to Y$ are
the projections.
Consequently,
\[
 \delta_{X\times Y}(p_X^*D+p_Y^*E)=\delta_X(D)+\delta_Y(E).
\]
In particular, the usual anticanonical indices are additive whenever
$-K_X$ and $-K_Y$ are nef.
\end{proposition}

\begin{proof}
For complete toric varieties,
$\Pic(X\times Y)=\Pic(X)\oplus\Pic(Y)$ and
$\Gamma_{X\times Y}=\Gamma_X\oplus\Gamma_Y$.
By Lemma~\ref{lem:hilbert-basic}~(iv),
the Hilbert basis of
$\Gamma_{X\times Y}$ is
$(\cH_X\times\{0\})\cup(\{0\}\times\cH_Y)$.
A decomposition of $(D,E)\in\Gamma_X\oplus\Gamma_Y$
into these elements is equivalent
to a pair of Hilbert-basis decompositions of $D$ and $E$,
and its coefficient sum is the sum of their coefficient sums.
Proposition~\ref{prop:hilbert-formulation} proves the assertion.
\end{proof}

\begin{proposition}[Flop invariance]\label{prop:flop-invariance}
Let $X\dashrightarrow X^+$ be an elementary toric flop
between smooth projective toric varieties, and let $R$
be the extremal ray on $X$ defining the flopping contraction.
Identify $\Pic(X)$ and
$\Pic(X^+)$ by strict transform.
If $0\ne D\in\Nef(X)\cap\Pic(X)$
and $D\cdot R=0$, then its strict transform $D^+$ is nef and
\[
\tau_X(D;\mathbb K)=\tau_{X^+}(D^+;\mathbb K)
\qquad (\mathbb K\in\{\ZZ,\QQ\}).
\]
In particular, if $X$ and $X^+$ are weak Fano, then
\[
\tau_X(\mathbb K)=\tau_{X^+}(\mathbb K)
\qquad (\mathbb K\in\{\ZZ,\QQ\}),
\qquad
\delta(X)=\delta(X^+).
\]
\end{proposition}

\begin{proof}
Let $\gamma$ be the numerical class of the extremal primitive
relation
\[
x_1+\cdots+x_s=y_1+\cdots+y_s
\qquad (s\ge2)
\]
corresponding to the flop, so that $R=\RR_{\ge0}\gamma$.
The smoothness of both varieties makes all nonzero coefficients
of the primitive relation equal to $1$ in absolute value,
and crepancy makes the two sides have the same number of terms.
Set $w:=x_1+\cdots+x_s=y_1+\cdots+y_s$.
Introducing the ray generated by $w$ gives a common
smooth toric variety $W$ with blow-ups
$p:W\to X$ and $q:W\to X^+$ along
$V(\langle y_1,\ldots,y_s\rangle)$ and
$V(\langle x_1,\ldots,x_s\rangle)$, respectively;
see, for example, \cite[Definition~6.18]{sato-towardfano}.
Let $E$ be their common exceptional divisor on $W$.

For a torus-invariant divisor $A=\sum_v c_vD_v$ on $X$,
let $A^+$ denote its strict transform on $X^+$,
where $D_v$ stands for the torus-invariant prime divisor
associated to a $1$-dimensional cone $\RR_{\ge 0}v$.
The pullback formulas give
\begin{equation}\label{eq:flop-pullback}
p^*A-q^*A^+
=
\left(\sum_{j=1}^s c_{y_j}
      -\sum_{j=1}^s c_{x_j}\right)E
=-(A\cdot\gamma)E.
\end{equation}
Thus \eqref{eq:flop-pullback} gives $p^*A=q^*A^+$
whenever $A\cdot\gamma=0$.
For such a class $A$, nefness can be checked after pullback
by the proper surjective morphisms $p$ and $q$.
Hence $A$ is nef if and only if $A^+$ is nef.

Fix $\mathbb K\in\{\ZZ,\QQ\}$ and let $D=\sum_i a_iD_i$
be an admissible decomposition.
The equality $0=D\cdot \gamma=\sum_i a_i(D_i\cdot \gamma)$,
together with the nefness of each $D_i$ and the positivity
of each $a_i$, implies $D_i\cdot \gamma=0$ for all $i$.
Hence each $D_i^+$ is nef.
Strict transform preserves nonzero integral classes,
so $D^+=\sum_i a_iD_i^+$ is an admissible decomposition
with the same coefficient sum.
The same argument applies to the inverse flop,
whose extremal relation has numerical class $-\gamma$.
Therefore $\tau_X(D;\mathbb K)=\tau_{X^+}(D^+;\mathbb K)$.
The anticanonical assertion follows by taking $D=-K_X$.
\end{proof}

\section{Generalized Bott manifolds and splitting fans}\label{sec:bott}

An $r$-stage generalized Bott tower is a sequence
\[
 B_r\longrightarrow B_{r-1}\longrightarrow\cdots\longrightarrow B_1\longrightarrow B_0=\{\mathrm{pt}\},
\]
where
\[
 B_p=\PP_{B_{p-1}}
 \bigl(\mathcal O_{B_{p-1}}\oplus L_p^{(1)}\oplus\cdots\oplus L_p^{(n_p)}\bigr)
\]
for line bundles $L_p^{(k)}$ on $B_{p-1}$.  The top variety $B_r$ is called a generalized Bott manifold.  If $n_p=1$ for every $p$, it is a Bott manifold.

Suyama gives an explicit fan description \cite{suyama-generalized-bott}.  Put
\[
 N:=\bigoplus_{p=1}^r\bigoplus_{k=1}^{n_p}\ZZ e_p^k.
\]
The rays are partitioned into blocks
\[
 P_p=\{u_p^0,u_p^1,\ldots,u_p^{n_p}\},
 \qquad
 u_p^k=e_p^k\quad(k\geq1),
\]
where $u_p^0$ is the negative of the sum of the basis vectors in the $p$-th block, plus an integral linear combination of basis vectors in later blocks.  The primitive collections are exactly $P_1,\ldots,P_r$, and their primitive relations have the upper-triangular form
\begin{equation}\label{eq:bott-relation}
 \gamma_p:\quad
 u_p^0+u_p^1+\cdots+u_p^{n_p}
 =\sum_{q>p}\sum_{k=0}^{n_q}c_{p,q,k}u_q^k,
 \qquad c_{p,q,k}\in\ZZ_{\geq0}.
\end{equation}
Some coefficients on the right may be zero.
The corresponding curve classes $\gamma_1,\ldots,\gamma_r$
generate the Mori cone; see
\cite[Remark~5]{suyama-generalized-bott}.

For a smooth complete toric variety $X=X_\Sigma$, let
\[
 \cR_\Sigma
 :=\ker\left(\ZZ^{\Sigma(1)}\longrightarrow N,
              \ e_\rho\longmapsto v_\rho\right).
\]
By the toric divisor sequence and its dual, $\cR_\Sigma$ is naturally identified with the integral numerical curve lattice
$\N_1(X)_\ZZ=\Hom(\Pic(X),\ZZ)$.

For a primitive collection $P$ with relation
$x_1+\cdots+x_\ell=\sum_j a_jy_j$ and
the corresponding numerical class $\gamma$,
we write $\deg(P)=\ell-\sum_j a_j=(-K_X)\cdot\gamma$;
we also denote this integer by $\deg(\gamma)$.

\begin{theorem}[Unimodularity for generalized Bott manifolds]\label{thm:bott-unimodular}
Let $X$ be a generalized Bott manifold.  Then the primitive-relation classes
$\gamma_1,\ldots,\gamma_r$
form a $\ZZ$-basis of $\N_1(X)_\ZZ$ and generate $\NEbar(X)$.  Consequently, the Mori cone is unimodular with respect to $\N_1(X)_\ZZ$, and the nef cone is unimodular with respect to $\Pic(X)$.
\end{theorem}

\begin{proof}
The rays $u_p^k$ with $k\geq1$ form a $\ZZ$-basis of $N$.  Let
\[
f:\mathbb Z^{\Sigma(1)}
\longrightarrow
\mathbb Z^{\Sigma(1)}
\Big/
\left\langle e_{u_p^k}\mid 1\le p\le r,\ 1\le k\le n_p\right\rangle
\simeq \mathbb Z^r
\]
be the quotient map, where the last isomorphism is induced by the classes of
\(e_{u_1^0},\ldots,e_{u_r^0}\),
which form a basis of the quotient.
We denote by
$\pi:=f|_{\cR_\Sigma}:\cR_\Sigma\longrightarrow \ZZ^r$
its restriction.
Equivalently, \(\pi\) records the coefficients of
\(u_1^0,\ldots,u_r^0\) in an integral ray relation.

This map is injective: a relation in the kernel of \(\pi\) is a relation among
the rays \(u_p^k\) with \(k\ge1\), which form a
\(\mathbb Z\)-basis of \(N\);
hence the relation is zero.
It is also surjective. Indeed,
for any $(a_1,\ldots,a_r)\in\ZZ^r$,
$a_1u_1^0+\cdots+a_ru_r^0\in N$ can be uniquely expressed
as an integral linear combination of the basis vectors
$u_p^k$ with $k\ge1$.
Moving this combination to the left gives a relation whose
image under $\pi$ is $(a_1,\ldots,a_r)$.
Therefore
$ \pi:\cR_\Sigma\xrightarrow{\ \sim\ }\ZZ^r$
is an isomorphism.

In the coordinates given by $\pi$, the vector $\gamma_p$ has coefficient $1$ in the $p$-th position, coefficient $0$ in every earlier position, and possibly nonzero integer coefficients only in later positions, by the upper-triangular relation \eqref{eq:bott-relation}.
Hence the matrix with rows $\pi(\gamma_1),\ldots,\pi(\gamma_r)$ is upper triangular with diagonal entries $1$. It has determinant $1$,
so $\gamma_1,\ldots,\gamma_r$ form a $\ZZ$-basis of $\cR_\Sigma=\N_1(X)_\ZZ$.

Thus the Mori cone is generated by a lattice basis.
Its dual cone is generated by the dual basis of $\Pic(X)$,
and is therefore unimodular as well.
\end{proof}

\begin{remark}[The Bott case]\label{rem:chary}
For Bott manifolds, the integral Mori-cone statement is proved in
\cite[Theorem~4.7 and Corollary~4.8]{chary-mori-bott}.
The corresponding dual description of the nef cone is given in
\cite[Theorem~5.7]{chary-mori-bott};
see also \cite{chary-mori-bott-erratum}.

\end{remark}

\begin{corollary}[Gongyo indices on generalized Bott manifolds]\label{cor:bott-total-index}
Let $X$ be a generalized Bott manifold and let
$0\neq D\in\Nef(X)\cap\Pic(X)$.  For $1\leq p\leq r$, put
$d_p(D):=D\cdot \gamma_p$.
Then $d_p(D)\in\ZZ_{\geq0}$ and
\[
 \tau_X(D;\ZZ)=\tau_X(D;\QQ)=\sum_{p=1}^r d_p(D),
 \qquad
 \delta_X(D)=0.
\]
In particular, if $X$ is weak Fano, then
\[
 \tau_X(\ZZ)=\tau_X(\QQ)=\sum_{p=1}^r\deg(P_p),
 \qquad \delta(X)=0.
\]
\end{corollary}

\begin{proof}
Let $L_1,\ldots,L_r$ be the basis of $\Pic(X)$ dual to
$\gamma_1,\ldots,\gamma_r$.
By Theorem~\ref{thm:bott-unimodular}, these are the
primitive nef-ray generators, and
\[
 D=d_1(D)L_1+\cdots+d_r(D)L_r.
\]
The assertion follows from Proposition~\ref{prop:unimodular-gap-zero}.  For
$D=-K_X$, one has $d_p(D)=\deg(P_p)$.
\end{proof}

We end this section by relating
the result to Batyrev's splitting fans.

\begin{definition}
A smooth complete fan is called a \emph{splitting fan} if any two distinct primitive collections are disjoint.
\end{definition}

\begin{proposition}[Splitting fans]\label{prop:splitting-bott}
A smooth complete toric variety has a splitting fan if and only if it is a generalized Bott manifold.  In particular, the nef cone of every smooth projective toric variety with a splitting fan is unimodular.
\end{proposition}

\begin{proof}
The equivalence follows from Batyrev's structure theorem for splitting fans
\cite[Theorem~4.3 and Corollary~4.4]{batyrev-tohoku}.
The final assertion follows from Theorem~\ref{thm:bott-unimodular}.
\end{proof}

\section{Smooth toric varieties of small Picard number}\label{sec:small-picard}

We next combine the preceding unimodularity theorem with the classifications
of Kleinschmidt and Batyrev.  The outcome is that Picard number at most three
forces a unimodular nef cone, even without a Fano assumption.

\subsection{Picard numbers one and two}

A $d$-dimensional smooth complete toric variety $X$
of Picard number one is the $d$-dimensional projective space.
Every nonzero nef integral class is $D=bH$ with $b\in\ZZ_{>0}$,
where $H$ is the hyperplane class on $\PP^d$.
Therefore
$\tau_X(D;\ZZ)=\tau_X(D;\QQ)=b$.
For $D=-K_X=(d+1)H$, this gives
$\tau_X(\ZZ)=\tau_X(\QQ)=d+1$.

For Picard number two, Kleinschmidt's classification shows that every smooth
complete toric variety $X$
of dimension $d$ is the projectivization of a decomposable vector
bundle over a projective space \cite{kleinschmidt-few-generators}.
It may be written as
\[
 X\simeq
 \PP_{\PP^s}\!\left(
   \mathcal O\oplus\mathcal O(a_1)\oplus\cdots\oplus\mathcal O(a_r)
 \right),
 \qquad
 0\leq a_1\leq\cdots\leq a_r,
\]
where $s,r\geq1$ and $d=s+r$.  We use the convention that $\PP(E)$
parametrizes one-dimensional quotients.  Let $H$ denote the pullback of the
hyperplane class on $\PP^s$ and let $\xi=\mathcal O_X(1)$.  The classes
$H,\xi$ form the primitive nef basis, and
\[
 -K_X=(s+1-a_1-\cdots-a_r)H+(r+1)\xi.
\]
Every nonzero nef integral class is uniquely of the form
$D=bH+c\xi$ with $b,c\in\ZZ_{\geq0}$ and $(b,c)\neq(0,0)$; hence
$\tau_X(D;\ZZ)=\tau_X(D;\QQ)=b+c$.
The variety is Fano precisely when $a_1+\cdots+a_r\leq s$, and it is weak
Fano precisely when $a_1+\cdots+a_r\leq s+1$.  Whenever $-K_X$ is nef,
its Gongyo indices are
$\tau_X(\ZZ)=\tau_X(\QQ)
 =d+2-\sum_{i=1}^r a_i$.
This also follows directly from Corollary~\ref{cor:bott-total-index}, since
the fan is splitting.

\subsection{Picard number three}
Here we use Batyrev's normal form to obtain the integral lattice
statement needed for the Gongyo indices.

Let $X=X_\Sigma$ be a smooth projective toric variety with $\rho(X)=3$.
Batyrev proved that $\Sigma$ has either three or five primitive collections
\cite[Theorem~5.7]{batyrev-tohoku}.  In the three-collection case the fan is
splitting, so Theorem~\ref{thm:bott-unimodular} applies.  We now treat the
five-collection case explicitly.

Following Batyrev's notation, partition the ray generators into five nonempty
sets
\[
 X_0=\{v_1,\ldots,v_{p_0}\},\quad
 X_1=\{y_1,\ldots,y_{p_1}\},\quad
 X_2=\{z_1,\ldots,z_{p_2}\},
\]
\[
 X_3=\{t_1,\ldots,t_{p_3}\},\qquad
 X_4=\{u_1,\ldots,u_{p_4}\},
 \qquad
 p_0+\cdots+p_4=d+3.
\]
The primitive collections are $X_i\cup X_{i+1}$, with indices taken modulo
five.  Up to a symmetry of the pentagon, their primitive relations are
\begin{align*}
 \gamma_0:\quad
 \sum_{i=1}^{p_0}v_i+\sum_{i=1}^{p_1}y_i
 &=\sum_{i=2}^{p_2}c_i z_i
   +\sum_{i=1}^{p_3}(b_i+1)t_i,\\
 \gamma_1:\quad
 \sum_{i=1}^{p_1}y_i+\sum_{i=1}^{p_2}z_i
 &=\sum_{i=1}^{p_4}u_i,\\
 \gamma_2:\quad
 \sum_{i=1}^{p_2}z_i+\sum_{i=1}^{p_3}t_i
 &=0,\\
 \gamma_3:\quad
 \sum_{i=1}^{p_3}t_i+\sum_{i=1}^{p_4}u_i
 &=\sum_{i=1}^{p_1}y_i,\\
 \gamma_4:\quad
 \sum_{i=1}^{p_4}u_i+\sum_{i=1}^{p_0}v_i
 &=\sum_{i=2}^{p_2}c_i z_i
   +\sum_{i=1}^{p_3}b_i t_i,
\end{align*}
where all $b_i$ and $c_i$ are nonnegative integers; an empty sum is
understood to be zero \cite[Theorem~6.6]{batyrev-tohoku}.

\begin{proposition}[The five-collection case]\label{prop:rho3-five}
In the notation above, the classes $\gamma_0,\gamma_1,\gamma_3$ form a $\ZZ$-basis of
$\N_1(X)_\ZZ$, and
$\NEbar(X)=\RR_{\geq0}\gamma_0+\RR_{\geq0}\gamma_1+\RR_{\geq0}\gamma_3$.
Consequently, both the Mori cone and the nef cone are unimodular.  If
$L_0,L_1,L_3$ is the nef basis dual to $\gamma_0,\gamma_1,\gamma_3$,
then every nonzero nef
integral class $D$ satisfies
\[
 \tau_X(D;\ZZ)=\tau_X(D;\QQ)
 =D\cdot \gamma_0+D\cdot \gamma_1+D\cdot \gamma_3.
\]
If $X$ is weak Fano, then in particular
\begin{equation}\label{eq:rho3-index-explicit}
\begin{aligned}
 \tau_X(\ZZ)=\tau_X(\QQ)
 &=\deg(\gamma_0)+\deg(\gamma_1)+\deg(\gamma_3)\\
 &=p_0+p_1+p_2-\sum_{i=2}^{p_2}c_i-\sum_{i=1}^{p_3}b_i.
\end{aligned}
\end{equation}
\end{proposition}

\begin{proof}
Batyrev's normal form also states that
\[
 \mathcal B=
 \{v_1,\ldots,v_{p_0},
   y_2,\ldots,y_{p_1},
   z_2,\ldots,z_{p_2},
   t_1,\ldots,t_{p_3},
   u_2,\ldots,u_{p_4}\}
\]
is a $\ZZ$-basis of $N$.  Therefore the map
$\pi:\cR_\Sigma\longrightarrow\ZZ^3$
that records the coefficients of the three omitted rays $y_1,z_1,u_1$ is an
isomorphism: arbitrary coefficients on these three rays can be canceled
uniquely by an integral combination of the basis rays in $\mathcal B$.
For the three displayed relations one has
\[
 \pi(\gamma_0)=(1,0,0),\qquad
 \pi(\gamma_1)=(1,1,-1),\qquad
 \pi(\gamma_3)=(-1,0,1).
\]
The matrix with these rows has determinant one.
Hence $\gamma_0,\gamma_1,\gamma_3$ form a
$\ZZ$-basis of $\cR_\Sigma=\N_1(X)_\ZZ$.

Moreover, the five relations satisfy
\[
 \gamma_2=\gamma_1+\gamma_3,
 \qquad
 \gamma_4=\gamma_0+\gamma_3
\]
in the relation lattice.  Since the primitive-relation classes generate the
Mori cone of a smooth projective toric variety
\cite[Theorem~2.15]{batyrev-tohoku}, the asserted description of
$\NEbar(X)$ follows.  Its generators form a lattice basis, so the Mori cone
and its dual nef cone are unimodular.

Let $L_0,L_1,L_3$ be the nef basis dual to
$\gamma_0,\gamma_1,\gamma_3$.  For a nef
integral class $D$, its coefficients in this basis are
$D\cdot \gamma_0$, $D\cdot \gamma_1$, $D\cdot \gamma_3$,
so the general formula follows from
Proposition~\ref{prop:unimodular-gap-zero}.
If $-K_X$ is nef, then
\[
 -K_X=\deg(\gamma_0)L_0+\deg(\gamma_1)L_1+\deg(\gamma_3)L_3,
\]
and the coefficients are nonnegative.
Finally,
\[
 \deg(\gamma_0)
 =p_0+p_1-\sum_{i=2}^{p_2}c_i
  -\sum_{i=1}^{p_3}b_i-p_3,
\]
\[
 \deg(\gamma_1)=p_1+p_2-p_4,
 \qquad
 \deg(\gamma_3)=p_3+p_4-p_1.
\]
Adding these equalities gives \eqref{eq:rho3-index-explicit}.
\end{proof}

\begin{theorem}[Small Picard number]\label{thm:small-picard}
Let $X$ be a smooth projective toric variety with $\rho(X)\leq3$.  Then
$\Nef(X)$ is unimodular with respect to $\Pic(X)$.  Consequently, for every
$0\neq D\in\Nef(X)\cap\Pic(X)$,
\[
 \tau_X(D;\ZZ)=\tau_X(D;\QQ),
 \qquad \delta_X(D)=0.
\]
In particular, the anticanonical Gongyo indices agree whenever $X$ is weak
Fano.
\end{theorem}

\begin{proof}
The Picard-number-one and Picard-number-two cases were treated above.  If
$\rho(X)=3$, Batyrev's dichotomy gives either a splitting fan, handled by
Theorem~\ref{thm:bott-unimodular}, or the five-collection case, handled by
Proposition~\ref{prop:rho3-five}.  The assertion for every $D$ follows from
Proposition~\ref{prop:unimodular-gap-zero}.
\end{proof}

\section{Toric models and Fano lifts}\label{sec:ST}

Let $d\geq3$ and $2\leq n\leq d+1$.  Let $B_n^d$ be the blow-up of $\PP^d$
at $n$ torus-invariant points.  In \cite{sato-tsuzuki-antiflips}, smooth
projective toric models are constructed by a prescribed sequence of
anti-flips, and the existence of a smooth Fano model is determined.
Choose ray generators
\[
 x_1,\ldots,x_{d+1}\in N\simeq\ZZ^d,
 \qquad x_1+\cdots+x_{d+1}=0,
\]
with $x_1,\ldots,x_d$ a basis, and put
\[
 y_i=-x_i\qquad(1\leq i\leq n).
\]
The models in Subsections~\ref{subsec:ST-bundle}
and~\ref{subsec:ST-critical} have the same primitive ray generators
$x_1,\ldots,x_{d+1}$ and $y_1,\ldots,y_n$.

\subsection{The bundle branch: \texorpdfstring{$2n-1<d$}{2n-1<d}}\label{subsec:ST-bundle}

In this range the final model $\tB_n^d$ has primitive relations
\begin{align}
 x_i+y_i&=0 &&(1\leq i\leq n),\label{eq:ST-low-I}\\
 x_{n+1}+\cdots+x_{d+1}&=y_1+\cdots+y_n.\label{eq:ST-low-II}
\end{align}
The primitive collections are pairwise disjoint, so $\tB_n^d$ is a
$(\PP^1)^n$-bundle over $\PP^{d-n}$ and has a splitting fan.

\begin{proposition}\label{prop:ST-low}
Assume $2n-1<d$.  Then $\Nef(\tB_n^d)$ is unimodular.
Consequently,
$\delta_{\tB_n^d}(D)=0$
for every nonzero nef integral divisor $D$.
In particular,
$\tau_{\tB_n^d}(\ZZ)=\tau_{\tB_n^d}(\QQ)=d+1$.
\end{proposition}

\begin{proof}
The unimodularity and the assertion for arbitrary $D$ follow from
Proposition~\ref{prop:splitting-bott} and
Corollary~\ref{cor:bott-total-index}.  The degrees of
\eqref{eq:ST-low-I} and \eqref{eq:ST-low-II} are $2$ and
$d-2n+1$, respectively, whose sum with multiplicity is
$2n+(d-2n+1)=d+1$.
\end{proof}

\subsection{The critical range in even and odd dimensions}\label{subsec:ST-critical}

Write
\[
 d=2m+\varepsilon,
 \qquad \varepsilon\in\{0,1\},
\]
and assume $2n-1\geq d$, equivalently
\[
 m+1\leq n\leq2m+\varepsilon+1.
\]
Let
\[
 Z_{m,n}^{\varepsilon}:=B_{n,m}^{2m+\varepsilon}
\]
denote the model at the critical stage of the construction in
\cite{sato-tsuzuki-antiflips}.  For
$\varepsilon=0$ this is the Fano model $\tB_n^{2m}$; for
$\varepsilon=1$ it is a smooth weak Fano model.
For $m\ge2$, we use the standard notation
\[
\widetilde V^{2m}:=Z^0_{m,2m},
\qquad
V^{2m}:=Z^0_{m,2m+1}.
\]
These are the toric pseudo-del Pezzo and del Pezzo varieties,
respectively.
The primitive relations of $Z_{m,n}^{\varepsilon}$ are as follows;
see \cite[Proposition~3.1, Lemma~3.2, and Theorem~3.3]{sato-tsuzuki-antiflips}.
\begin{align}
 x_i+y_i&=0 &&(1\leq i\leq n),\label{eq:ST-critical-I}\\
 \sum_{j\notin I}x_j&=\sum_{i\in I}y_i
 &&\bigl(I\subset\{1,\ldots,n\},\ |I|=m\bigr),
 \label{eq:ST-critical-II}\\
 \sum_{j\in J}y_j&=\sum_{\ell\notin J}x_\ell
 &&\bigl(J\subset\{1,\ldots,n\},\ |J|=m+1\bigr).
 \label{eq:ST-critical-III}
\end{align}
Complements are taken in the full set
$\{1,\ldots,2m+\varepsilon+1\}$ of $x$-indices.  The degrees are
$2$, $1+\varepsilon$, and $1-\varepsilon$, respectively.  Thus the even
model is Fano, while the odd model has precisely the degree-zero relations
in \eqref{eq:ST-critical-III}.

Write an invariant divisor as
\[
 D=\sum_{j=1}^{2m+\varepsilon+1}a_jD_{x_j}
   +\sum_{i=1}^n b_iD_{y_i},
\]
and put
\[
 q_i:=a_i+b_i\quad(1\leq i\leq n),
 \qquad
 A:=\sum_{j=1}^{2m+\varepsilon+1}a_j.
\]

\begin{lemma}[Integral coordinates]\label{lem:ST-coordinates}
The assignment $[D]\mapsto(q_1,\ldots,q_n;A)$ identifies
$\Pic(Z_{m,n}^{\varepsilon})$ with $\ZZ^{n+1}$.
\end{lemma}

\begin{proof}
Adding the principal divisor associated to $u\in M$ changes
$a_j$ by $\langle u,x_j\rangle$ and $b_i$ by
$-\langle u,x_i\rangle$, so the $q_i$ and $A$ are invariant
(note that $\sum_{j=1}^{2m+\varepsilon+1}x_j=0$).
If they all vanish, then $b_i=-a_i$ and
$\sum_{j=1}^{d+1} a_j=0$.
Since $x_1,\ldots,x_d$ form a $\ZZ$-basis of $N$,
there exists $u\in M$ such that
$\langle u,x_j\rangle=a_j$ for $1\le j\le d=2m+\varepsilon$.
Then
\[
\langle u,x_{d+1}\rangle
=-\sum_{j=1}^d a_j=a_{d+1},
\]
and, since $y_i=-x_i$,
\[
\langle u,y_i\rangle=-a_i=b_i.
\]
Hence $D=\operatorname{div}(\chi^u)$ is principal.
Conversely,
arbitrary integers $(q_1,\ldots,q_n;A)$ are obtained by choosing the $a_j$
with sum $A$ and putting $b_i=q_i-a_i$.
\end{proof}

\begin{lemma}[The common nef semigroup]\label{lem:ST-nef}
A class $(q_1,\ldots,q_n;A)$ is nef if and only if
\begin{align}
 q_i&\geq0 &&(1\leq i\leq n),\label{eq:ST-nef1}\\
 \sum_{i\in I}q_i&\leq A &&(|I|=m),\label{eq:ST-nef2}\\
 A&\leq\sum_{j\in J}q_j &&(|J|=m+1).\label{eq:ST-nef3}
\end{align}
Here $I$ and $J$ range over subsets of $\{1,\ldots,n\}$
of cardinalities $m$ and $m+1$, respectively.
Moreover,
\[
 -K_{Z_{m,n}^{\varepsilon}}
 =(2,\ldots,2;2m+\varepsilon+1).
\]
For fixed $m\ge2$ and $m+1\le n\le2m+1$, the even and odd
critical models have the same nef semigroup under these coordinates,
while their anticanonical classes differ.
\end{lemma}

\begin{proof}
The intersections with the three types of primitive relations are
$q_i$, $A-\sum_{i\in I}q_i$, and
$\sum_{j\in J}q_j-A$, respectively.  These relation classes generate the
Mori cone, so their nonnegativity is equivalent to nefness.  The formula for
$-K$ follows by assigning coefficient one to every invariant prime divisor.
\end{proof}

The boundary member $n=m+1$ is unimodular in both parities.

\begin{proposition}\label{nismplusone}
For $n=m+1$, put
\[
S=\sum_{i=1}^{m+1} q_i,\qquad t=S-A,\qquad r_i=q_i-t.
\]
Then
\[
\operatorname{Nef}(Z^\varepsilon_{m,m+1})
\cap \operatorname{Pic}(Z^\varepsilon_{m,m+1})
\simeq \ZZ_{\ge0}^{m+2}
\]
with free coordinates $(t,r_1,\ldots,r_{m+1})$.
Consequently every nonzero nef integral class has zero gap.
For the anticanonical class,
\[
\bigl(
\tau_{Z^\varepsilon_{m,m+1}}(\ZZ),
\tau_{Z^\varepsilon_{m,m+1}}(\QQ)
\bigr)
=
\begin{cases}
(m+2,m+2), & \varepsilon=0,\\
(2m+2,2m+2), & \varepsilon=1.
\end{cases}
\]
\end{proposition}

\begin{proof}
Since $n=m+1$, condition~\eqref{eq:ST-nef3} is simply
$A\le S$, hence $t\ge0$.
Moreover, every $m$-subset of $\{1,\ldots,m+1\}$
is obtained by omitting one index, so condition~\eqref{eq:ST-nef2}
is equivalent to
\[
S-q_i\le A
\qquad (1\le i\le m+1),
\]
or equivalently $t\le q_i$.
Thus $r_i=q_i-t\ge0$ for all $i$.

Conversely, if $t,r_1,\ldots,r_{m+1}\in\ZZ_{\ge0}$, then
\[
q_i=t+r_i,\qquad
A=mt+\sum_{i=1}^{m+1} r_i,
\]
and these satisfy the nef conditions of Lemma~\ref{lem:ST-nef}.
The change of coordinates and its inverse are integral linear maps.
Thus they identify the full Picard lattice with $\ZZ^{m+2}$
and the nef semigroup with $\ZZ_{\ge0}^{m+2}$.
The equality of the integral and rational Gongyo indices
for every nonzero nef integral class follows from
Proposition~\ref{prop:unimodular-gap-zero}.

For the anticanonical class one has $q_i=2$ for all $i$
and $A=2m+\varepsilon+1$.
Hence
\[
S=2m+2,\qquad
t=1-\varepsilon,
\qquad
r_i=1+\varepsilon.
\]
The stated values of the two indices follow.
\end{proof}

For $n\geq m+2$, put
\[
 r:=n-m-1\geq1,
 \qquad
 C(q;A):=\sum_{i=1}^nq_i-A.
\]
We also regard $A$ as the last-coordinate linear functional,
and abbreviate $C(q;A)$ to $C$ when the class is clear.
Taking complements in $\{1,\ldots,n\}$, condition~\eqref{eq:ST-nef3}
is equivalent to
\begin{equation}\label{eq:nef-C}
 \sum_{i\in T}q_i\le C
 \qquad(T\subset\{1,\ldots,n\},\ |T|=r).
\end{equation}

\begin{lemma}\label{lem:ST-positive-support}
If $(q;A)$ is a nonzero nef integral class and
$n\ge m+2$, then
\[
A\ge m,\qquad C\ge r.
\]
\end{lemma}

\begin{proof}
Suppose first that at most $m$ of the $q_i$ are positive.
Choose an $m$-subset $I\subset\{1,\ldots,n\}$ containing
all indices for which $q_i>0$. Since all $q_i$ are nonnegative,
condition~\eqref{eq:ST-nef2} gives
\[
 A\ge\sum_{i\in I}q_i=\sum_{i=1}^nq_i,
 \qquad
 C=\sum_{i=1}^nq_i-A\le0.
\]
On the other hand, \eqref{eq:nef-C} gives
$\sum_{i\in T}q_i\le C$ for every $r$-subset $T$.
Since the left-hand side is nonnegative, we obtain $C=0$,
and hence every such sum is zero.
As $r\ge1$, this forces $q_i=0$ for all $i$, and then
$A=0$, contradicting the assumption that $(q;A)$ is nonzero.
Therefore at least $m+1$ of the $q_i$ are positive.

Since the $q_i$ are integers, an $m$-subset consisting of
positive entries has sum at least $m$.
Condition~\eqref{eq:ST-nef2} therefore gives $A\ge m$.

To prove $C\ge r$, suppose that at most $r$ of the $q_i$
are positive. An $r$-subset containing all positive entries
then gives $C\ge\sum_iq_i$ by \eqref{eq:nef-C}.
Since $A+C=\sum_iq_i$, this implies $A\le0$.
Condition~\eqref{eq:ST-nef2}, together with $q_i\ge0$,
forces $A=0$ and $q_i=0$ for all $i$, a contradiction.
Thus at least $r+1$ entries are positive.
An $r$-subset of positive entries has sum at least $r$,
so \eqref{eq:nef-C} gives $C\ge r$.
\end{proof}

\begin{theorem}[Anticanonical indices in the critical range]
\label{thm:ST-critical}
Assume $n\geq m+2$.
\begin{enumerate}[label=\textup{(\roman*)}]
 \item For $\varepsilon=0$ and $m\geq2$,
 \[
  \bigl(\tau_{Z_{m,n}^{0}}(\ZZ),
        \tau_{Z_{m,n}^{0}}(\QQ)\bigr)
  =\left(2,\frac{2m+1}{m}\right),
  \qquad \delta(Z_{m,n}^{0})=\frac1m.
 \]
 \item For $\varepsilon=1$,
 \[
  \bigl(\tau_{Z_{m,n}^{1}}(\ZZ),
        \tau_{Z_{m,n}^{1}}(\QQ)\bigr)=(2,2),
  \qquad \delta(Z_{m,n}^{1})=0.
 \]
\end{enumerate}
\end{theorem}

\begin{proof}
For $1\le i\le n$, let $L_i$ be the class with $q_i=0$,
$q_j=1$ for $j\ne i$, and $A=m$, and put
$M_-=(1,\ldots,1;m)$ and $M_+=(1,\ldots,1;m+1)$.
These are nonzero nef integral classes by Lemma~\ref{lem:ST-nef}.
In the even case ($\varepsilon=0$),
\[
 m(-K_{Z_{m,n}^{0}})
 =\sum_{i=1}^nL_i+(2m-n+1)M_-,
\]
where $2m-n+1=m-r\ge0$ since $n\le2m+1$ in the even branch.
This gives a rational decomposition of
length $(2m+1)/m$.  Lemma~\ref{lem:ST-positive-support} shows that
$\lambda_A(q;A)=A/m$ is at least one on every nonzero nef integral class,
so Corollary~\ref{cor:dual} gives the matching upper bound.
Finally,
$-K=M_-+M_+$, and $2<(2m+1)/m<3$, proving the integral value.

In the odd case ($\varepsilon=1$),
\[
 -K_{Z_{m,n}^{1}}=2M_+.
\]
Thus the integral index is at least two.
Since $C(-K)=2n-(2m+2)=2r$ and $\lambda_C(q;A)=C/r$ is
at least one on every nonzero nef
integral class by Lemma~\ref{lem:ST-positive-support}, the rational index is
at most two by Corollary~\ref{cor:dual}.  Hence both values equal two.
\end{proof}

\begin{remark}[The two-dimensional endpoint]\label{rem:m1}
Although dimension two is outside the standing assumption
$d\ge3$ of this section, the corresponding toric
pseudo-del Pezzo and del Pezzo surfaces
$\widetilde V^2$ and $V^2$ satisfy
\[
\tau(\ZZ)=\tau(\QQ)=3.
\]
$\widetilde V^2$ and $V^2$ are del Pezzo surfaces of
degree $7$ and $6$, respectively.
\end{remark}

\begin{corollary}[del Pezzo and pseudo-del Pezzo varieties]\label{cor:delpezzo}
For $m\geq2$,
\[
 \tB_{2m}^{2m}=\tV^{2m},
 \qquad
 \tB_{2m+1}^{2m}=V^{2m},
\]
and
\[
 \tau_{\tV^{2m}}(\ZZ)=\tau_{V^{2m}}(\ZZ)=2,
 \qquad
 \tau_{\tV^{2m}}(\QQ)=\tau_{V^{2m}}(\QQ)=2+\frac1m.
\]
In particular,
\[
 \delta(\tV^{2m})=\delta(V^{2m})=\frac1m.
\]
\end{corollary}

\begin{proof}
The identifications are the terminal cases of the family described in
\cite[Remark~3.8]{sato-tsuzuki-antiflips}; the formulas follow from
Theorem~\ref{thm:ST-critical}.
\end{proof}

\begin{corollary}[The odd branch and its flops]\label{cor:ST-odd-flops}
Every standard odd-dimensional critical model $Z_{m,n}^{1}$ has zero
Gongyo-index gap.  The same is true for every smooth projective weak Fano
model obtained from it by a sequence of toric flops.
\end{corollary}

\begin{proof}
The standard models have zero gap by Proposition~\ref{nismplusone}
when $n=m+1$ and by Theorem~\ref{thm:ST-critical} when $n\ge m+2$.
The assertion for flops follows from Proposition~\ref{prop:flop-invariance}
by iteration.
\end{proof}

Thus moving among the weak Fano flop chambers cannot create a positive gap.
We next change the anticanonical class by a dimension-raising construction.

\subsection{A dimension-raising construction and the anticanonical shift}
\label{subsec:dimension-raising}

We recall the construction introduced in
\cite[Section~4 and Proposition~4.1]{sato-index-two}.
Let $X=X_\Sigma$ be a smooth projective toric $d$-fold,
and let $\G(\Sigma)$ denote the set of primitive ray generators
of $\Sigma$.
Fix $x\in \G(\Sigma)$ and an integer $p\ge2$.
Put $\widehat N=N\oplus\ZZ^{p-1}$ and identify $N$
with the first summand.
Let $z_1,\ldots,z_{p-1}$ be the standard basis vectors
of the second summand, and set
\[
z_p:=x-z_1-\cdots-z_{p-1}.
\]
The fan $\widehat\Sigma$ of the construction has
primitive ray generators
\[
\G(\widehat\Sigma)
=
(\G(\Sigma)\setminus\{x\})\cup\{z_1,\ldots,z_p\}.
\]
Its primitive collections are precisely
\[
\widehat P=
\begin{cases}
P, & x\notin P,\\
(P\setminus\{x\})\cup\{z_1,\ldots,z_p\}, & x\in P,
\end{cases}
\]
where $P$ ranges over the primitive collections of $\Sigma$.
We write $H(x,p)(X):=X_{\widehat\Sigma}$.
This is a smooth projective toric variety of dimension
$d+p-1$.

The primitive relation associated to $\widehat P$
is obtained from that of $P$ by replacing $x$
with $z_1+\cdots+z_p$, whether $x$ occurs on the left
or on the right.
Relations not involving $x$ remain unchanged.
Indeed, a cone of $\Sigma$ not containing $x$ remains a cone
of $\widehat\Sigma$. A cone containing $x$ gives a cone of
$\widehat\Sigma$ by replacing $x$ with the whole block
$z_1,\ldots,z_p$.
Apply this to the cone containing the sum of the elements of $P$
in its relative interior. A positive coefficient of $x$ on the
right becomes the same positive coefficient on every $z_i$.
Thus the substituted right-hand side still lies in the relative
interior of the corresponding cone, as required for the
primitive relation.

\begin{proposition}[Nef semigroup and anticanonical shift]\label{prop:H-shift}
Put $\widehat X=H(x,p)(X)$.
For $v\in\G(\Sigma)\setminus\{x\}$, let $\widehat D_v$
denote the corresponding invariant prime divisor on $\widehat X$.
There is a natural lattice isomorphism
\[
 \Phi_x:\Pic(X)\xrightarrow{\sim}\Pic(\widehat X)
\]
such that
\begin{gather*}
 \Phi_x([D_v])=[\widehat D_v]\quad(v\in\G(\Sigma)\setminus\{x\}),\\
 \Phi_x([D_x])=[D_{z_1}]=\cdots=[D_{z_p}],
\end{gather*}
and
\[
 \Phi_x(\Nef(X)\cap\Pic(X))
 =\Nef(\widehat X)\cap\Pic(\widehat X).
\]
Moreover,
\begin{equation}\label{eq:sato-shift}
 \Phi_x^{-1}(-K_{\widehat X})=-K_X+(p-1)[D_x].
\end{equation}
Consequently, whenever $\widehat X$ is weak Fano,
\[
 \tau_{\widehat X}(\mathbb K)
 =\tau_X(-K_X+(p-1)[D_x];\mathbb K)
 \qquad(\mathbb K\in\{\ZZ,\QQ\}).
\]
\end{proposition}

\begin{proof}
We use the notation in the construction above.
Let $\eta_1,\ldots,\eta_{p-1}$ be the basis of the second
summand of $\widehat M=M\oplus(\ZZ^{p-1})^\vee$
dual to $z_1,\ldots,z_{p-1}$.
Since $z_p=x-z_1-\cdots-z_{p-1}$, we have
\[
\operatorname{div}_{\widehat X}(\chi^{\eta_i})
=D_{z_i}-D_{z_p}
\qquad (1\le i\le p-1).
\]
Thus $[D_{z_1}]=\cdots=[D_{z_p}]$.
For $u\in M$, regarded as $(u,0)\in\widehat M$, we also have
\[
\operatorname{div}_{\widehat X}(\chi^{(u,0)})
=
\sum_{v\ne x}\langle u,v\rangle\widehat D_v
+\langle u,x\rangle D_{z_p},
\]
since $\langle (u,0),z_i\rangle=0$ for $1\le i\le p-1$,
and $\langle (u,0),z_p\rangle=
\langle (u,0),x-z_1-\ldots-z_{p-1}\rangle
=\langle u,x\rangle$.
These are precisely the principal relations on $X$ with
$D_x$ replaced by $D_{z_p}$.
Since $M$ and the $\eta_i$ generate $\widehat M$,
the displayed principal relations generate all principal
relations among the invariant divisors on $\widehat X$.
The toric divisor sequence therefore gives the lattice
isomorphism $\Phi_x$ with the stated values.

We next compare nefness.
Write a primitive relation of $X$, with all terms moved
to the left, as $\gamma:\sum_v b_vv=0$.
In the corresponding primitive relation $\widehat\gamma$
of $\widehat X$, every $z_i$ has coefficient $b_x$,
and the coefficients of the other rays are unchanged.
For an invariant divisor $A=\sum_v a_vD_v$ on $X$,
the class $\Phi_x([A])$ is represented by
$\widehat A=\sum_{v\ne x}a_v\widehat D_v+a_xD_{z_p}$.
Hence
\[
\widehat A\cdot\widehat\gamma
=
\sum_{v\ne x}a_vb_v+a_xb_x
=
A\cdot\gamma.
\]
Since the primitive-relation classes generate the Mori cone on each
variety, $A$ is nef if and only if $\widehat A$ is nef;
see \cite[Theorem~2.15]{batyrev-tohoku} and
\cite[Theorem~1.4 and Proposition~1.10]{cox-vonrenesse-primitive}.
Together with the lattice isomorphism, this proves the
asserted identification of the nef semigroups.

Finally, using $[D_{z_i}]=\Phi_x([D_x])$ for all $i$, we obtain
\[
\Phi_x^{-1}(-K_{\widehat X})
=
\sum_{v\ne x}[D_v]+p[D_x]
=
-K_X+(p-1)[D_x].
\]
The equality of the Gongyo indices follows by transporting
decompositions through $\Phi_x$, which preserves nonzero
nef integral classes and their coefficients.
\end{proof}

Formula~\eqref{eq:sato-shift} explains the use of Gongyo indices
with respect to arbitrary nef divisors:
$H(x,p)$ keeps the semigroup fixed and changes the target.

\subsection{Minimal raywise lifts of the odd branch}

Fix the standard odd model
\[
 Y_{m,n}:=Z_{m,n}^{1}=B_{n,m}^{2m+1}.
\]
Assign weights $\alpha_i\in\ZZ_{\geq0}$ to $y_i$ and
$\beta_j\in\ZZ_{\geq0}$ to $x_j$.  A positive weight means that we apply
$H(y_i,\alpha_i+1)$ or $H(x_j,\beta_j+1)$, respectively, to the original
ray; no operation is applied recursively to newly created rays.
Since this construction replaces each chosen ray independently
by a block of new rays whose sum is the original ray, these
operations commute up to toric isomorphism. Thus the resulting
variety depends only on the weights $\alpha_i$ and $\beta_j$,
not on the order in which the operations are performed.
We call the resulting variety a \emph{raywise lift} of $Y_{m,n}$.
Put
\begin{equation}\label{eq:lift-weights}
\begin{gathered}
 W:=\sum_{i=1}^n\alpha_i+\sum_{j=1}^{2m+2}\beta_j,
 \qquad B:=\sum_{j=1}^{2m+2}\beta_j,\\
 u_i:=\alpha_i+\beta_i\qquad(1\le i\le n).
\end{gathered}
\end{equation}
The dimension of the resulting variety is $2m+1+W$.

\begin{proposition}[Fano criterion]\label{prop:ST-lift-criterion}
The raywise lift is Fano if and only if
\begin{equation}\label{eq:lift-Fano}
 \max_{|I|=m}\sum_{i\in I}u_i
 \le B+1\le
 \min_{|J|=m+1}\sum_{j\in J}u_j,
\end{equation}
where $I,J\subset\{1,\ldots,n\}$.
\end{proposition}

\begin{proof}
By Proposition~\ref{prop:H-shift}, the nef cone is still described by
Lemma~\ref{lem:ST-nef}. Using \eqref{eq:sato-shift} and
\eqref{eq:lift-weights}, the shifted anticanonical class is
\begin{equation}\label{eq:lift-target}
(2+u_1,\ldots,2+u_n;\,2m+2+B).
\end{equation}
The variety is Fano if and only if this class lies in the
interior of the nef cone.
The inequalities $q_i>0$ are automatic.
For every $m$-subset $I$, condition~\eqref{eq:ST-nef2} becomes
\[
\sum_{i\in I}(2+u_i)<2m+2+B,
\]
which, since all quantities are integral, is equivalent to
\[
\sum_{i\in I}u_i\le B+1.
\]
Similarly, for every $(m+1)$-subset $J$, condition~\eqref{eq:ST-nef3}
becomes
\[
2m+2+B<\sum_{j\in J}(2+u_j),
\]
or equivalently
\[
B+1\le\sum_{j\in J}u_j.
\]
Taking the maximum over all $m$-subsets $I$ and the minimum
over all $(m+1)$-subsets $J$ gives the stated criterion.
\end{proof}

The case $n=m+1$ is exceptional.  Applying $H(y_c,2)$
for one index $c$ gives
\[
u_c=1,\qquad u_i=0\ (i\ne c),\qquad B=0.
\]
Hence both sides of the Fano criterion in Proposition~\ref{prop:ST-lift-criterion}
are equal to $1$, so the resulting variety is Fano.
Its anticanonical class has coordinates
\[
(q;A)=(q_1,\ldots,q_{m+1};2m+2),
\]
where $q_c=3$ and $q_i=2$ for $i\ne c$.
In the free coordinates of Proposition~\ref{nismplusone} this becomes
\[
t=1,\qquad r_c=2,\qquad r_i=1\quad(i\ne c).
\]
Therefore
\[
\tau(\ZZ)=\tau(\QQ)=m+3
\]
by Proposition~\ref{prop:unimodular-gap-zero}.
We henceforth assume $n\ge m+2$.

\begin{theorem}[Minimal Fano lift]\label{thm:ST-minimal-lift}
Assume $m+2\leq n\leq2m+2$.  Among raywise lifts of the standard model
$Y_{m,n}$, the total weight of a Fano lift satisfies
\[
 W\geq n-1.
\]
Equality holds exactly as follows, up to permuting
the first $n$ indices.
Choose $T\subset\{2,\ldots,n\}$ with $|T|=m-1$,
and put $S=\{2,\ldots,n\}\setminus T$, so $|S|=n-m$.
Apply
\[
H(x_i,2)\quad(i\in T),
\qquad
H(y_j,2)\quad(j\in S),
\]
and leave all other rays unchanged.
The resulting variety,
denoted $\widehat Y_{m,n}$, is a smooth toric Fano variety of dimension
\[
 \dim\widehat Y_{m,n}=2m+n.
\]
\end{theorem}

\begin{proof}
Consider a Fano lift.
After permuting the first $n$ indices and relabeling
the rays and weights accordingly, we may assume that
$u_1\le\cdots\le u_n$.
The Fano criterion~\eqref{eq:lift-Fano} gives
\[
\sum_{i=n-m+1}^n u_i
\le B+1
\le \sum_{i=1}^{m+1}u_i.
\]
Since $B\ge0$, not all $u_i$ are zero, so $u_n>0$.

Suppose that $u_1=u_2=0$.
Since $n\ge m+2$, the indices $3,\ldots,m+1,n$
are distinct and form an $m$-subset.
Hence
\[
\sum_{i=n-m+1}^n u_i
\ge \sum_{i=3}^{m+1}u_i+u_n
> \sum_{i=1}^{m+1}u_i,
\]
contradicting the Fano criterion.
Here the sum from $3$ to $m+1$ is empty when $m=1$.
Thus at most one $u_i$ is zero.
Since the $u_i$ are nonnegative integers, we obtain
\[
W
=\sum_{i=1}^n u_i+\sum_{j=n+1}^{2m+2}\beta_j
\ge \sum_{i=1}^n u_i
\ge n-1.
\]

Suppose now that $W=n-1$.
Equality in the preceding inequalities forces
\[
u_1=0,\qquad
u_i=1\quad(2\le i\le n),\qquad
\beta_j=0\quad(j>n).
\]
The Fano criterion therefore reduces to
$m\le B+1\le m$, so $B=m-1$.
Since $u_1=\alpha_1+\beta_1=0$, we have
$\alpha_1=\beta_1=0$.
For every $2\le i\le n$, the equality
$\alpha_i+\beta_i=1$ implies that exactly one of
$\alpha_i,\beta_i$ is $1$ and the other is $0$.
Put
\[
T=\{i\in\{2,\ldots,n\}\mid\beta_i=1\},
\qquad
S=\{2,\ldots,n\}\setminus T.
\]
Since $\beta_j=0$ for $j>n$, $|T|=B=\sum_j\beta_j=m-1$ and $|S|=n-m$.
Thus the operations are exactly
$H(x_i,2)$ for $i\in T$ and $H(y_j,2)$ for $j\in S$,
with no operation on the remaining rays.
This is the stated construction.

Conversely, this construction gives
$u_1=0$, $u_i=1$ for $2\le i\le n$, and $B=m-1$.
Both inequalities in the Fano criterion hold with equality,
so the resulting variety is Fano.
Its total weight is $W=n-1$.
There are $n-1$ operations, each of which increases
the dimension by one, and hence
\[
\dim\widehat Y_{m,n}
=(2m+1)+(n-1)
=2m+n.
\]
Permutations of $x_1,\ldots,x_n$, together with the corresponding
$y_i$, are induced by lattice automorphisms of the original fan.
In particular, permutations fixing the first index act transitively
on the subsets $T\subset\{2,\ldots,n\}$ of size $m-1$.
The constructions therefore give isomorphic toric varieties,
so $\widehat Y_{m,n}$ is well defined up to toric isomorphism.
\end{proof}

\begin{remark}\label{rem:iterated-sato}
The word ``minimal'' refers to lifts of the standard model $Y_{m,n}$.
It does not compare constructions performed after a preliminary flop.
Operations on newly created rays do not enlarge the class of lifts:
if $p,q\ge2$ and $z_i$ belongs to the block replacing $x$
in $H(x,p)(X)$, then
\[
 H(z_i,q)\bigl(H(x,p)(X)\bigr)\simeq H(x,p+q-1)(X).
\]
To see this, permute the new block so that $i=p$ and write
$z_p=x-z_1-\cdots-z_{p-1}$.
If $t_1,\ldots,t_{q-1}$ are the additional basis vectors,
the second construction replaces $z_p$ by these vectors and
\[
 t_q=x-z_1-\cdots-z_{p-1}-t_1-\cdots-t_{q-1}.
\]
These are exactly the generators of a single block of size $p+q-1$.
The two successive replacements of primitive collections agree
with the single replacement, and the lattice bases agree as well.
Permuting a block is a unimodular change of basis fixing $N$.
Thus the isomorphism also holds for any $i$.
By iteration, all descendants of a given original ray can be
grouped into one block, with the same total dimension increase.

Consequently, restricting attention to operations on the original rays
causes no loss of generality among iterated applications of this
construction without intervening flops.
\end{remark}

\begin{theorem}[Gongyo indices of the minimal lift]
\label{thm:ST-lift-indices}
Let $m\ge1$ and $m+2\le n\le2m+2$, and let
$\widehat Y_{m,n}$ be the minimal Fano lift in
Theorem~\ref{thm:ST-minimal-lift}.
Then
\[
\bigl(
\tau_{\widehat Y_{m,n}}(\ZZ),
\tau_{\widehat Y_{m,n}}(\QQ)
\bigr)
=
\begin{cases}
(4,4),
  & (m,n)=(1,3),\\[2pt]
\left(3,3+\dfrac1m\right),
  & m\ge2,\quad m+2\le n\le2m+1,\\[6pt]
\left(3,3+\dfrac1{m+1}\right),
  & n=2m+2.
\end{cases}
\]
In particular, every such minimal lift has positive
Gongyo-index gap except when $(m,n)=(1,3)$.
\end{theorem}

\begin{proof}
We use the coordinates $(q;A)$ of $Y_{m,n}$ and
the identification of nef semigroups induced by
Proposition~\ref{prop:H-shift}.
By Theorem~\ref{thm:ST-minimal-lift}, we may take
$u_1=0$, $u_i=1$ for $2\le i\le n$, and $B=m-1$.
Thus, by \eqref{eq:lift-target}, the anticanonical class of
$\widehat Y_{m,n}$ corresponds to
\begin{equation}\label{eq:minimal-target}
 b=(2,3,\ldots,3;\,3m+1).
\end{equation}
For each $1\le i\le n$, let $L_i$ be the class with
$q_i=0$, $q_j=1$ for $j\ne i$, and $A=m$.
Put
\[
M_-=(1,\ldots,1;m),
\qquad
M_+=(1,\ldots,1;m+1).
\]
These are nonzero nef integral classes by Lemma~\ref{lem:ST-nef}.
Since
\[
b=M_-+M_++L_1,
\]
we have $\tau_{\widehat Y_{m,n}}(\ZZ)\ge3$.
We also note that
\[
\sum_{i=1}^nL_i=(n-1,\ldots,n-1;\,nm).
\]

Suppose first that $n\le2m+1$.
Comparison with the coordinates in \eqref{eq:minimal-target} gives
\[
b=\frac1m\sum_{i=1}^nL_i
  +\frac{2m-n+1}{m}M_-+L_1.
\]
All coefficients are nonnegative, and their sum is
$n/m+(2m-n+1)/m+1=3+1/m$.
On the other hand, Lemma~\ref{lem:ST-positive-support}
gives $A\ge m$
for every nonzero nef integral class.
Thus Corollary~\ref{cor:dual}, applied to the linear functional
$A/m$, gives
\[
\tau_{\widehat Y_{m,n}}(\QQ)
\le\frac{A(b)}m
=\frac{3m+1}{m}.
\]
Hence $\tau_{\widehat Y_{m,n}}(\QQ)=3+1/m$.

Suppose next that $n=2m+2$.
In this case, comparison of coordinates gives
\[
b=\frac1{m+1}\sum_{i=1}^nL_i
  +\frac1{m+1}M_++L_1.
\]
The coefficient sum is $(n+1)/(m+1)+1=3+1/(m+1)$.
Recall that $C=\sum_iq_i-A$.
Lemma~\ref{lem:ST-positive-support} now gives $C\ge m+1$
for every nonzero nef integral class, while
\[
C(b)=2+3(n-1)-(3m+1)=3m+4.
\]
Applying Corollary~\ref{cor:dual} to $C/(m+1)$ therefore gives
\[
\tau_{\widehat Y_{m,n}}(\QQ)
\le\frac{3m+4}{m+1}
=3+\frac1{m+1}.
\]
Thus equality holds.

Except when $(m,n)=(1,3)$, the rational index is
strictly between $3$ and $4$.
Since the integral index is an integer and is at least
$3$, it equals $3$.
When $(m,n)=(1,3)$, the first rational decomposition
above becomes
\[
b=2L_1+L_2+L_3,
\]
an integral decomposition of coefficient sum four.
Hence both indices are four in this case.
\end{proof}

\begin{example}[A six-dimensional positive-gap example]
The smallest-dimensional positive-gap member of the
family in Theorem~\ref{thm:ST-lift-indices} is $\widehat Y_{1,4}$.
Start with $Y_{1,4}=B^3_{4,1}$, the blow-up of $\PP^3$
at its four torus-fixed points, and apply $H(y_i,2)$
for $i=2,3,4$.
For each such $i$, denote the two primitive ray generators
replacing $y_i$ by $z_i,w_i$, so that $z_i+w_i=y_i$.
The resulting sixfold has primitive ray generators
\[
x_1,\ldots,x_4,\ y_1,\ z_2,w_2,\ z_3,w_3,\ z_4,w_4.
\]
Its fourteen primitive relations are
\[
\begin{aligned}
x_1+y_1&=0,\\
x_i+z_i+w_i&=0
  &&(2\le i\le4),\\
x_2+x_3+x_4&=y_1,\\
\sum_{\substack{1\le j\le4\\j\ne i}}x_j
  &=z_i+w_i
  &&(2\le i\le4),\\
y_1+z_i+w_i
  &=\sum_{j\in\{2,3,4\}\setminus\{i\}}x_j
  &&(2\le i\le4),\\
z_i+w_i+z_j+w_j&=x_1+x_k
  &&(\{i,j,k\}=\{2,3,4\},\ i<j).
\end{aligned}
\]
These are obtained from the primitive relations of
$Y_{1,4}$ by the replacements $y_i=z_i+w_i$
for $i=2,3,4$.
All of them have positive degree.
By Theorem~\ref{thm:ST-lift-indices},
\[
\tau_{\widehat Y_{1,4}}(\ZZ)=3,
\qquad
\tau_{\widehat Y_{1,4}}(\QQ)=\frac72,
\qquad
\delta(\widehat Y_{1,4})=\frac12.
\]
\end{example}

\section{Pseudo-symmetric varieties and the main theorem}\label{sec:pseudosym}

A smooth toric Fano variety $X=X_\Sigma$ of dimension $d$ is called \emph{pseudo-symmetric} if its fan contains a pair of centrally symmetric maximal cones: there exists $\sigma\in\Sigma(d)$ such that $-\sigma\in\Sigma(d)$.
Ewald's structure theorem \cite{ewald-fano}, in the form of
\cite[Theorem~6.6]{sato-towardfano}, gives a product decomposition
of this class.

\begin{theorem}[Ewald]\label{thm:ewald}
Let $X$ be a pseudo-symmetric smooth toric Fano variety.
Then there exist nonnegative integers $s,a,b,p,q$ and integers
$m_\alpha\geq2$ for $1\leq\alpha\leq p$ and
$n_\beta\geq2$ for $1\leq\beta\leq q$ such that
\[
 X\simeq(\PP^1)^s\times(V^2)^a\times(\tV^2)^b
 \times\prod_{\alpha=1}^p V^{2m_\alpha}
 \times\prod_{\beta=1}^q\tV^{2n_\beta}.
\]
\end{theorem}

The anticanonical calculations above make both Gongyo indices explicit.

\begin{corollary}\label{cor:pseudosym-formula}
Let $X$ have a decomposition as in Theorem~\ref{thm:ewald}.
Then
\begin{align*}
 \tau_X(\ZZ)
 &=2s+3(a+b)+2(p+q),\\
 \tau_X(\QQ)
 &=2s+3(a+b)+2(p+q)
   +\sum_{\alpha=1}^p\frac1{m_\alpha}
   +\sum_{\beta=1}^q\frac1{n_\beta},
\end{align*}
and therefore
\[
 \gap(X)
 =\sum_{\alpha=1}^p\frac1{m_\alpha}
  +\sum_{\beta=1}^q\frac1{n_\beta}.
\]
\end{corollary}

\begin{proof}
For $\mathbb K\in\{\ZZ,\QQ\}$, one has $\tau_{\PP^1}(\mathbb K)=2$,
while Remark~\ref{rem:m1} gives
$\tau_{V^2}(\mathbb K)=\tau_{\tV^2}(\mathbb K)=3$.
Corollary~\ref{cor:delpezzo} gives the indices of all
factors of dimension at least four.
Adding the contributions of the factors by
Proposition~\ref{prop:product} proves the formulas.
\end{proof}

\begin{theorem}[Main theorem]\label{thm:main}
Let $X$ be a pseudo-symmetric smooth toric Fano variety of dimension
$d\ge4$. If $\gap(X)>0$, then
\[
 \gap(X)\geq\frac1{\lfloor d/2\rfloor}.
\]
Thus Conjecture~\ref{conj:minimal-positive-gap} holds for pseudo-symmetric smooth toric Fano varieties.

Under the assumption $\gap(X)>0$, equality holds as follows:
\begin{enumerate}[label=\textup{(\roman*)}]
 \item if $d=2m$, equality holds if and only if
 \[
  X\simeq V^{2m}\quad\text{or}\quad X\simeq\tV^{2m};
 \]
 \item if $d=2m+1$, equality holds if and only if
 \[
  X\simeq V^{2m}\times\PP^1
  \quad\text{or}\quad
  X\simeq\tV^{2m}\times\PP^1.
 \]
\end{enumerate}
\end{theorem}

\begin{proof}
Take the decomposition in Theorem~\ref{thm:ewald}, and list
$m_1,\ldots,m_p,n_1,\ldots,n_q$ as $r_1,\ldots,r_k$,
where $k=p+q$.
Each $r_i\geq2$, and $k\geq1$ since $\gap(X)>0$.
Corollary~\ref{cor:pseudosym-formula} gives
\[
 \gap(X)=\sum_{i=1}^k\frac1{r_i}.
\]
Put $R=r_1+\cdots+r_k$.  Since $r_i\leq R$,
\[
 \gap(X)=\sum_{i=1}^k\frac1{r_i}
 \geq\frac1R.
\]
The decomposition also gives
\[
 d=s+2a+2b+2R,
\]
so $R\leq\lfloor d/2\rfloor$.
Hence
\[
 \gap(X)\geq\frac1R
 \geq\frac1{\lfloor d/2\rfloor}.
\]

Suppose equality holds.  If $k\geq2$, then
\[
 \sum_{i=1}^k\frac1{r_i}\geq\frac{k}{R}>\frac1R,
\]
so $k=1$.  Equality in the second inequality forces
\[
 r_1=R=\left\lfloor\frac d2\right\rfloor.
\]
If $d=2m$, then $R=m$ and $s+2a+2b=0$,
so $s=a=b=0$ and $X$ is $V^{2m}$ or $\tV^{2m}$.
If $d=2m+1$, then $R=m$ and $s+2a+2b=1$,
which forces $s=1$ and $a=b=0$.
Thus $X$ is $V^{2m}\times\PP^1$ or
$\tV^{2m}\times\PP^1$.
The converse follows from Corollary~\ref{cor:delpezzo}
and Proposition~\ref{prop:product}.
\end{proof}

The preceding theorem gives a sharp lower bound in each fixed
dimension within the pseudo-symmetric class.
When the dimension is allowed to vary, however, every positive
rational number occurs as a gap.

\begin{corollary}[Every positive rational occurs]\label{cor:any-rational}
For every $r\in\QQ_{>0}$, there exists a pseudo-symmetric
smooth toric Fano variety $X$ such that
\[
 \gap(X)=r.
\]
\end{corollary}

\begin{proof}
Write $r=p/q$ with $p,q\in\ZZ_{>0}$.
If $q\ge2$, take $X=(V^{2q})^p$.
By Corollary~\ref{cor:delpezzo} and
Proposition~\ref{prop:product}, we have $\gap(X)=p/q=r$.
If $q=1$, take $X=(V^4)^{2p}$; then $\gap(X)=p=r$.
Both constructions are pseudo-symmetric.
\end{proof}

\appendix

\section{Exact calculations in dimensions at most seven}\label{sec:computations}
We record the computational evidence that motivated
Conjecture~\ref{conj:minimal-positive-gap}.
The cases of dimensions one and two are treated separately below.
For dimensions three through six, we used the database in the Macaulay2
package \texttt{NormalToricVarieties}, with zero-based database indices
\cite{batyrev-fano-fourfolds,sato-towardfano,macaulay2-normal-toric}.
For dimension seven, we used Paffenholz's vertex-format lists of the
$72{,}256$ smooth reflexive polytopes obtained from \O bro's classification
\cite{obro-classification,paffenholz-fano-data}.
The input files are \texttt{fano-v7d-0.tgz} through
\texttt{fano-v7d-7.tgz}.
For an input polytope $Q$, the toric fan is the face fan of its polar dual
$Q^*$. The original archive-member names and zero-based global indices
are retained in the computational supplement. The seven-dimensional source
vertex arrays and their full fan representations are not redistributed;
a supplied retrieval script obtains the original archives directly from
the distribution site and checks their recorded SHA-256 values.
We also checked that the $866$ five-dimensional input fans are pairwise
non-isomorphic as lattice fans.

For each variety in dimensions three through seven, we computed the fan,
the Picard lattice, the nef cone, and the Hilbert basis of
$\Nef(X)\cap\Pic(X)$ using Macaulay2, Normaliz, and SageMath
\cite{macaulay2-system,normaliz,sagemath}.
The saved export metadata record Macaulay2~1.26.06, and the original
pipeline documentation specifies SageMath~10.8.
A separate installation snapshot dated 18~September~2026 records
\texttt{NormalToricVarieties}~1.9, the Macaulay2
\texttt{Normaliz} interface~2.6, \texttt{Polyhedra}~1.10,
and the Normaliz executable~3.11.1.
The supplement distinguishes these later observations from the original
run metadata and records the remaining dependencies there.

The rational optimum was obtained from an exact linear computation,
and the integral optimum from an exact integer-hull computation or a
finite multiset search. Both values admit short exact certificates.
Writing $\cH_X=\{h_1,\ldots,h_N\}$ and $b=[-K_X]$, each saved record
contains coefficients $x_j\in\QQ_{\ge0}$ and $z_j\in\ZZ_{\ge0}$,
and a rational linear functional $\lambda$, satisfying
\begin{equation}\label{eq:computation-certificate}
\begin{gathered}
 b=\sum_jx_jh_j=\sum_jz_jh_j,\qquad
 \lambda(h_j)\ge1\quad(1\le j\le N),\\
 \sum_jx_j=\lambda(b),\qquad
 \sum_jz_j=\lfloor\lambda(b)\rfloor.
\end{gathered}
\end{equation}
Thus Corollary~\ref{cor:dual} proves that
$\tau_X(\QQ)=\lambda(b)$; the integral witness and
$\tau_X(\ZZ)\le\lfloor\tau_X(\QQ)\rfloor$ give
$\tau_X(\ZZ)=\lfloor\lambda(b)\rfloor$ for every record in these lists.
All equalities and inequalities in \eqref{eq:computation-certificate}
were independently rechecked over $\QQ$, including all
$72{,}256$ seven-dimensional records.
This certificate verification uses the supplied Hilbert bases;
reconstructing the nef cones and verifying Hilbert-basis completeness
are separate steps. The code for these steps is supplied; in dimension seven
it uses the separately retrieved inputs.
The replay code was also tested on five selected cases in dimensions
four, six and seven, reconstructing their nef cones and Hilbert bases
from the saved fans and reproducing both indices. This was a
selected-case test, not a full-census reconstruction.

The computational supplement~\cite{gongyo-computation-data}
(version~1.0.2) contains the computation and verification code,
exact certificates, source identifiers and checksums, together
with the input fans in dimensions three through six.
The seven-dimensional inputs are obtained separately using
the supplied retrieval script.
The README gives separate instructions for checking the
numerical certificates, reconstructing the geometric data,
and regenerating the tables.

\subsection{The elementary cases in dimensions one and two}

In dimension one, the only smooth toric Fano variety is $\PP^1$.  If $H$ is
the positive generator of $\Pic(\PP^1)$, then $-K_{\PP^1}=2H$ and
$\Nef(\PP^1)\cap\Pic(\PP^1)=\ZZ_{\geq0}H$.
Hence
$\left(\tau_{\PP^1}(\ZZ),\tau_{\PP^1}(\QQ)\right)=(2,2)$.

The five smooth toric Fano surfaces and their Gongyo indices
are listed in Table~\ref{tab:surface-indices}; see
\cite[Theorem~5.7]{enwright-et-al-nef-complexity}.
The blow-up centres in the table are distinct torus-fixed points.
\begin{center}
\begin{minipage}{\linewidth}
\normalfont
\captionof{table}{Gongyo indices of smooth toric del Pezzo surfaces.}
\label{tab:surface-indices}
\centering
\begin{tabular}{>{$}c<{$} >{$}c<{$}}
\toprule
 X&\bigl(\tau_X(\ZZ),\tau_X(\QQ)\bigr)\\
\midrule
 \PP^2&(3,3)\\
 \PP^1\times\PP^1&(4,4)\\
 \operatorname{Bl}_{p_1,\ldots,p_r}\PP^2,\quad 1\le r\le3&(3,3)\\
\bottomrule
\end{tabular}
\end{minipage}
\end{center}
In particular, every smooth toric Fano surface has zero Gongyo-index gap.
Notice also that
$\operatorname{Bl}_{p_1,p_2,p_3}\PP^2$
has a non-simplicial nef cone, so zero gap does not imply simpliciality even
in dimension two.

\subsection{The complete classifications through dimension seven}

\begin{theorem}[Computations through dimension seven]\label{thm:low-dimensional-computation}
For the complete classifications of smooth toric Fano varieties in
dimensions one through seven, Table~\ref{tab:classification} holds.
Here
``simplicial/unimodular'' counts varieties with a simplicial nef cone; every
such cone in the data is in fact unimodular.  The final column counts
varieties with zero gap and a non-simplicial, hence non-unimodular, nef cone.
\begin{center}
\begin{minipage}{\linewidth}
\normalfont\small
\captionof{table}{Gongyo-index gaps and nef cones through dimension seven.}
\label{tab:classification}
\centering
\setlength{\tabcolsep}{3.5pt}
\begin{tabular}{c r r l r r}
\toprule
$d$ & total & $\gap>0$ & \shortstack{positive-gap\\distribution}
& \shortstack{simplicial/\\unimodular}
& \shortstack{$\gap=0$,\\non-simplicial}\\
\midrule
1 & 1      & 0   & none                                      & 1      & 0\\
2 & 5      & 0   & none                                      & 4      & 1\\
3 & 18     & 0   & none                                      & 16     & 2\\
4 & 124    & 2   & $2$ of gap $1/2$                         & 101    & 21\\
5 & 866    & 8   & $8$ of gap $1/2$                         & 702    & 156\\
6 & $7{,}622$   & 81  & $75$ of gap $1/2$, $6$ of gap $1/3$     & $5{,}891$   & $1{,}650$\\
7 & $72{,}256$  & 727 & $684$ of gap $1/2$, $43$ of gap $1/3$   & $53{,}903$  & $17{,}626$\\
\bottomrule
\end{tabular}
\end{minipage}
\end{center}
In particular:
\begin{enumerate}[label=\textup{(\roman*)}]
 \item no positive gap occurs in dimensions at most three;
 \item the only positive gaps occurring in dimensions at most seven are
 $1/2$ and $1/3$;
 \item every gap example in dimensions at most seven has a non-simplicial
 nef cone, whereas every simplicial nef cone occurring in these
 classifications is unimodular;
 \item non-unimodularity is far from sufficient for a gap: there are already
 zero-gap examples with non-simplicial nef cone in dimension two, and there
 are $17{,}626$ such examples in dimension seven.
\end{enumerate}
\end{theorem}

Altogether, the seven classifications contain $80{,}892$ varieties.  The two
Gongyo indices agree for $80{,}074$ of them and differ for $818$, or about
$1.01\%$.  Among the latter, $769$ have gap $1/2$ and $49$ have gap $1/3$.
The smallest positive gap is therefore $1/2$ in dimensions four and five and
$1/3$ in dimensions six and seven, in agreement with the lower bound proposed
in Conjecture~\ref{conj:minimal-positive-gap}.

The pairs $(\tau_X(\ZZ),\tau_X(\QQ))$ occurring among the positive-gap
examples are listed in Table~\ref{tab:gap-pairs}.
\begin{center}
\begin{minipage}{\linewidth}
\normalfont
\captionof{table}{Index pairs among the positive-gap examples.}
\label{tab:gap-pairs}
\centering
$\begin{array}{c|c|r}
 d&(\tau_X(\ZZ),\tau_X(\QQ))&\text{number}\\
\hline
 4&(2,5/2)&2\\
\hline
 5&(2,5/2)&1\\
  &(3,7/2)&5\\
  &(4,9/2)&2\\
\hline
 6&(2,7/3)&6\\
  &(2,5/2)&3\\
  &(3,7/2)&30\\
  &(4,9/2)&27\\
  &(5,11/2)&13\\
  &(6,13/2)&2\\
\hline
 7&(2,7/3)&12\\
  &(2,5/2)&6\\
  &(3,10/3)&25\\
  &(3,7/2)&160\\
  &(4,13/3)&6\\
  &(4,9/2)&268\\
  &(5,11/2)&166\\
  &(6,13/2)&69\\
  &(7,15/2)&13\\
  &(8,17/2)&2
\end{array}$
\end{minipage}
\end{center}
The two four-dimensional examples are $\tV^4$ and $V^4$.  In dimension five,
the database indices of the eight gap examples are
$269,270,271,273,441,445,446,618$.
In dimension six, the six examples of gap $1/3$ include the three
models from Subsection~\ref{subsec:ST-critical}
\[
 \tB_5^6,\qquad \tB_6^6=\tV^6,\qquad \tB_7^6=V^6.
\]

\begin{remark}[Structure of the positive-gap examples]\label{rem:gap-structures}
The saved fans also exhibit a recurrent geometric pattern.
Among the eight five-dimensional positive-gap examples, six are
$V^4$- or $\tV^4$-bundles over $\PP^1$, including the two direct products.
The remaining two are of the form $H(x,2)(Y)$ with $Y$ a smooth
projective toric weak Fano fourfold which is not Fano.
In dimension six, $59$ of the $81$ examples admit toric bundle
presentations with lower-dimensional positive-gap fibres. Of the remaining $22$, twenty are obtained by
$H(x,2)$ from smooth projective toric weak Fano fivefolds which are not
Fano, and the other two are $V^6$ and $\tV^6$.
In dimension seven, $593$ of the $727$ examples admit toric bundle
presentations, including $131$ direct products; lower-dimensional
positive-gap fibres also occur widely among these presentations.
Each of the remaining
$134$ is obtained by $H(x,2)$ from a smooth projective toric weak Fano
sixfold which is not Fano.

Here ``bundle'' means an equivariantly locally trivial toric fibre
bundle, and the counts refer to varieties, not to choices of bundle
structure.
A bundle presentation was checked using a partition of the
rays for which the maximal cones form a combinatorial product
and the fibre rays span a real subspace of the expected dimension.
The fibre generators in a maximal cone form part of a lattice
basis, so the corresponding fibre lattice is saturated.
For an $H(x,2)$ presentation, we reconstructed the source fan
and checked the lattice coordinates, all maximal cones, and the support
inequalities for projectivity and nefness of the anticanonical class.
The corresponding finite certificates are included in the supplement;
verification in dimension seven uses the separately retrieved input data.
These observations concern only the classifications through dimension
seven; the bundle and $H(x,2)$ descriptions need not be mutually exclusive.
They do not assert additivity of Gongyo indices for nontrivial bundles.
\end{remark}

The nef-cone counts in Table~\ref{tab:classification} motivate the
following question. Simpliciality alone does not imply unimodularity
with respect to the lattice, but every simplicial nef cone in these
computations is unimodular.

\begin{question}\label{question:simplicial}
Let $X$ be a smooth toric Fano variety.  If $\Nef(X)$ is simplicial, must it
be unimodular with respect to $\Pic(X)$?  More weakly, must one have
$\delta(X)=0$?
\end{question}

\bibliographystyle{amsalpha}
\bibliography{ref}

\end{document}